\documentclass[12pt,a4paper,openany]{article}

\usepackage{amsmath,amssymb,amsfonts,amsthm}
\usepackage{graphics}
\usepackage{mathrsfs}
\usepackage[all]{xy}
\usepackage{color}
\usepackage{bm}
\usepackage{setspace}

\newtheorem{thm}{Theorem}[section]
\newtheorem{prop}[thm]{Proposition}
\newtheorem{lemma}[thm]{Lemma}

\newtheorem{cor}[thm]{Corollary}

\newtheorem{problem}[thm]{Problem}

\theoremstyle{definition}
\newtheorem{definition}[thm]{Definition}
\newtheorem{rem}[thm]{Remark}

\newtheorem*{ack}{Acknowledgment}

\newcommand{\sq}{\hfill $\square$}
\newcommand{\kah}{K\"{a}hler}
\newcommand{\dol}{\sqrt{-1}\partial \overline{\partial}}

\newcommand{\hol}{H\"{o}lder}
\newcommand{\ma}{Monge-Amp$\grave{{\rm e}}$re}

\makeatletter
\def\address#1#2{\begingroup
\noindent\parbox[t]{16cm}{%
\small{\scshape\ignorespaces#1}\par\vskip1ex
\noindent\small{\itshape E-mail address}%
\/: #2\par\vskip4ex}\hfill%
\endgroup}%
\makeatother

\makeatletter

\@addtoreset{equation}{section}
\makeatother

\title{Microscopic stability thresholds and constant scalar curvature K\"{a}hler metrics}
\author{Takahiro Aoi}
\date{\today}
\begin{document}
\maketitle

\begin{abstract}
In this paper, we directly prove that if the limit of microscopic stability thresholds introduced by Berman for a polarized manifold satisfies some condition, then there exists a unique constant scalar curvature K\"{a}hler metric.
This is an analogue of K.Zhang's result which is proved by the delta-invariant introduced by Fujita-Odaka.
This work is motivated by Berman's result which shows that if a Fano manifold is uniformly Gibbs stable, then there exists a unique K\"{a}hler-Einstein metric, without uniform K-stability.
We also give some sufficient conditions of the existence of a constant scalar curvature K\"{a}hler cone metric.
\end{abstract}



\section{Introduction}

In {\kah} geometry, the existence of canonical {\kah} metrics is a central problem.
{\kah}-Einstein metrics and constant scalar curvature {\kah} metrics are known as candidates of canonical {\kah} metrics.
It is well-known as the solution of the (uniform) 
Yau-Tian-Donaldson conjecture that a Fano manifold admits a unique  {\kah}-Einstein metric with positive Ricci curvature if and only if it is (uniformly) K-stable \cite{BBJ,CDS1,CDS2,CDS3,Ti,Zha}.
From the statistical mechanics point of view, Berman proposed a probabilistic approach to {\kah}-Einstein metrics in order to find a canonical way to construct such metrics.
He also showed that there exists a canonical sequence of probability measures which converges the volume form defined by the {\kah}-Einstein metric with negative Ricci curvature \cite{Be3}(see also \cite{Be1}).
However, a similar approach is a subtle problem in the Fano case, because the corresponding partition function is not well-defined in general.
For a Fano manifold $X$, Berman \cite{Be3} introduced the microscopic stability threshold $\gamma_k (X)$ at level $k$ and the notion of uniform Gibbs stability by taking their limit.
By the purely algebro-geometric way, Fujita-Odaka \cite{FO} proved that if a Fano manifold $X$ is uniformly Gibbs stable, then $X$ is uniformly K-stable, i.e., $X$ admits a unique {\kah}-Einstein metric.
This result is proved by the inequality $\delta_k(X) \geq \gamma_k (X)$, where $\delta_k(X)$ denotes the delta-invariant at level $k$ introduced by Fujita-Odaka, and the fact that their limit $\delta (X) = \lim_k \delta_k (X)$ characterizes uniform K-stability \cite{BJ,FO}.
After their work, Berman \cite{Be2} gives a direct proof which says that uniform Gibbs stability implies the existence of a unique {\kah}-Einstein metric, without uniform K-stability and the delta-invariant $\delta (X)$.

In this paper, we give an direct proof which shows that the limit of the microscopic stability thresholds $\gamma_k(L)$ for a polarized manifold $(X,L)$ (not necessarily Fano) satisfies the condition in Theorem 2.4 of K.Zhang's paper \cite{Zha}, then there exists a unique constant scalar curvature {\kah} metric.
Actually, as an application of the delta-invariant $\delta(L)$ for $(X,L)$, we can prove the same result by combining Fujita-Odaka's result \cite[Theorem 2.5]{FO} (or \cite[Corollary 3.4]{Be2}) for $(X,L)$ and K.Zhang's result \cite[Theorem 2.4]{Zha}.
In our proof, we only use the explicit inequality for the Mabuchi K-energy, the partition function and $\mathcal{J}_\chi$-functional.
Our proof can be applied to the existence of constant scalar curvature {\kah} metric with cone singularities along a divisor by using K.Zheng's variational characterization \cite{Zhe}.

Let $X$ be a projective manifold and $L$ be an ample line bundle over $X$.
Fix a {\kah} metric $\omega$ defined by a Hermitian metric $h$ on $L$, i.e., $\omega =\sqrt{-1} \Theta (L,h)$.
Here the symbol $\Theta (L,h)$ denotes the Chern curvature form of $h$.
We denotes $\mathcal{H}(L)$ the space of {\kah} metrics in the first Chern class $c_1 (L)$.
For $k \in \mathbb{Z}_{>0}$, we fix a basis $s_{i}^{(k)}, (i =1,2,...,N)$ of $H^0 (X , L^{\otimes k})$, where $N = N_k := \dim_{\mathbb{C}} H^0 (X , L^{\otimes k}) $.
We define the following holomorphic section of the line bundle $(L^{\otimes k})^{ \boxtimes N}$ over the product space $X^N := X \times ... \times X$ ($N$-times), which is called the {\it Slater determinant} in the statistical mechanics :
$$
\det S^{(k)} (x_1 , ... , x_N) := \det \left( s_{i}^{(k)} (x_j) \right)_{ij} \in H^0 \left(X^N , (L^{\otimes k})^{ \boxtimes N} \right).
$$

We fix a smooth probability measure $dV$ on $X$.
Let $f \geq 0$ be a function in $L^p (dV)$ for some $p >1$ with $\int_X f dV =1$.
We define a (possibly singular) probability measure $dV_f$ by
$$
dV_f := f dV .
$$
For $\gamma >0$, we define the {\it $k$-th partition function for $(X, L, f)$} by
\begin{equation*}
\mathcal{Z}_{N, f} (-\gamma) := \int_{X^N} \left| \det S^{(k)} \right|^{-2\gamma/k}_{h^k} dV_f^{\otimes N}.
\end{equation*}
Here, the parameter $- \gamma$ corresponds to the (negative) inverse temperature (see \cite{Be1}).
Since we consider the integral of the negative power of the function $\left| \det S^{(k)} \right|^2$, the partition function $\mathcal{Z}_{N, f} (-\gamma)$ may diverge.
By following \cite[Section 6]{Be3} and \cite[Section 3]{Be1}, we define the {\it microscopic stability threshold at level $k$} by
\begin{equation*}
\gamma_{k,f} (L) = \gamma_{N_k,f} (L):= \sup \left\{ \gamma \geq 0  \,\, \middle| \,\, \mathcal{Z}_{N, f} (-\gamma) < \infty \right\}.
\end{equation*}
Note that $\gamma_{k, f} (L)$ is well-defined, i.e., the finiteness of the partition function is independent of the choice of a basis $ ( s_{i}^{(k)} )_{i =1}^{N}$.
By following \cite{Be3} (see also \cite{Be1}), we write the infimum limit of the microscopic stability thresholds as
\begin{equation*}
\gamma_{ f}(L) := \liminf_{N_k \to \infty} \gamma_{N_k, f} (L).
\end{equation*}
To simplify, we write $\gamma_{k}(L) = \gamma_{k,f} (L)$ and $\gamma (L) = \gamma_f (L)$ for $f =1$.
Note that we can write $\gamma_k(L)$ as the log canonical threshold in birational geometry \cite{Be1,Be3,FO} : $$\gamma_k (L) = {\rm lct}_{\Delta} \left( X^N ; \frac{1}{k} ( \det S^{(k)} =0)  \right),$$where $\Delta \subset X^N$ is the diagonal.

Let $\eta$ be a $d$-closed smooth (1,1)-form on $X$.
We define the {\it Mabuchi K-energy for $f$ and $\eta$} which is a modification of the Mabuchi functional introduced by Mabuchi \cite{Ma}, by the following formula.
$$
\mathcal{M}_{f, \eta} (\phi) := {\rm Ent}_{dV_f} (V^{-1} \omega_\phi^n) + \mathcal{J}_{-{\rm Ric } \, dV + \eta} (\phi) , \,\,\,\, \phi \in \mathcal{H} (L).
$$
Here, ${\rm Ent}_{dV_f}$ is the relative entropy for $dV_f$ and $\mathcal{J}_{-{\rm Ric} \, dV + \eta }$ is the $\mathcal{J}_\chi$-functional with respect to the smooth $d$-closed (1,1)-form $\chi = -{\rm Ric} \, dV + \eta$ (see Section 2 in this paper).

The following theorem is the main result in this paper, which is an extension of Theorem 1.2 in \cite{Be2}.
\begin{thm}
\label{main theorem}
We fix a positive number $c$ such that $c \, \omega + {\rm Ric}\,dV \geq 0$.
Fix $\gamma >0$ and $\tau >0$.
We take a positive integer $k$ such that $ f \in L^{k\tau / (k\tau-\gamma(1+\tau))}  (dV)$.
Then, there exists a positive constant $C_1$ independent of $k, \gamma, \tau$ such that we have
\begin{eqnarray*}
&& \hspace{-40pt} -\frac{1}{N(1+\tau)} \log \mathcal{Z}_{N, f} (-\gamma(1+\tau)) +  \mathcal{J}_{-{\rm Ric} \, dV + \eta  + \gamma^\prime \omega } (\phi)\\
&&\leq \mathcal{M}_{f, \eta}( \phi)  + \frac{k\tau-\gamma }{ k(1+\tau)} \log \int_{X} f^{k\tau / (k\tau-\gamma(1+\tau))}  dV+ \frac{\gamma }{k} \log N + C_1\gamma(1- ck^{-1})
\end{eqnarray*}
for any $\phi \in \mathcal{H}(L)$, where $\gamma^\prime := \gamma (1-ck^{-1})(1-C_1 k^{-1}) $.
\end{thm}
The parameter $\tau > 0$ is important to deal with the (possibly singular) density function $f \in L^p (dV)$.
By Theorem \ref{main theorem}, we immediately obtain the following.
\begin{thm}
\label{special}
We fix $\gamma \in (0, \gamma_{f} (L))$ so that $-c_1 (X) + [\eta] + \gamma c_1 (L) $ is a {\kah} class.
If $\mathcal{J}_{-{\rm Ric} \, dV +\eta + \gamma \omega }$ is coercive, then the Mabuchi K-energy $\mathcal{M}_{f, \eta}$ for $f$ and $\eta$ is also coercive.
\end{thm}
Note that the condition in this theorem is an analytic analogue for $\gamma_{f} (L)$ of the notion of special K-stability introduced by Hattori \cite[Definition 3.10]{Ha} (see also the solution of the (uniform) Lejmi-Sz\'{e}kelyhidi conjecture \cite[Conjecture 1]{LS} by G.Chen \cite{Gao}).

If $X$ is Fano, we can choose the reference data so that $\omega = {\rm Ric}\, dV$ without loss of generality.
In this case, we have $\mathcal{J}_{-{\rm Ric}\, dV + \gamma^\prime \omega }(\phi) =  (\gamma^\prime - 1) E^* (V^{-1} \omega_\phi^n)$ where $E^*$ is the pluricomplex energy (see \cite[Section 1]{BBEGZ}).
By setting $f=1$, Theorem \ref{main theorem} generalizes Theorem 1.2 in \cite{Be2} because the difference $\mathcal{M}_{1,0} - \mathcal{J}_{-{\rm Ric}\, dV + \gamma^\prime \omega }$ is exactly equal to the twisted Mabuchi functional for (negative) inverse temperature $-\gamma^\prime$ (equivalently, the free energy functional denoted by $F_{-\gamma^\prime }$).
By using the inequality in Theorem \ref{special}, we immediately obtain the following corollary which has been already proved by Fujita-Odaka and Berman \cite{Be2,FO}.
\begin{cor}[\cite{Be2,FO}]
Assume that $X$ is Fano and $L := -K_X$.
If $X$ is uniformly Gibbs stable, i.e., $\gamma (X) := \gamma (-K_X) >1$, then $X$ admits a unique {\kah}-Einstein metric.
\end{cor}
Here, the symbol $K_X$ denotes the canonical line bundle of $X$.

We can apply Theorem \ref{special} to the study of (possibly singular) constant scalar curvature {\kah} metrics.
Firstly, we set $f :=1$ and $\eta := 0$.
We define the following intersection number $\mu (L)$ which is proportional to the average value of the scalar curvature and the nef threshold $s(L)$ by
$$
\mu(L) := \frac{-K_X L^{n-1}}{L^n}, \,\,\, s (L) := \sup \left\{ s \in \mathbb{R} \, \middle| \, -K_X  - sL >0   \right\}.
$$
By using a sufficient condition of the coercivity of the $\mathcal{J}_\chi$-functional (\cite[Main Theorem]{We},\cite[Proposition 21,22]{CS}, see also \cite{Gao}) and a variational characterization of constant scalar curvature {\kah} metrics by Chen-Cheng and Berman-Darvas-Lu (\cite{CC2,BDL2}), we obtain the following corollary which is the $\gamma$-version of K.Zhang's result \cite[Theorem 2.4]{Zha}.

\begin{cor}
Assume that $K_X + \gamma (L) L$ is ample and $\gamma(L) > n \mu(L) - (n-1) s(L)$.
Then, there exists a unique constant scalar curvature {\kah} metric in $c_1 (L)$.
\end{cor}

\begin{rem}
Note that Fujita-Odaka \cite[Theorem 2.5]{FO} proved the inequality $\delta_k (L) \geq \gamma_k (L)$ for any $k$.
Moreover, Berman \cite[Corollary 3.4]{Be2} proved its analytic version, i.e., the analytic delta-invariant $\delta^A_k(L)$ at level $k$ which is the coercivity thresholds of the quantized Ding functional, satisfies $\delta^A_k(L) \geq \gamma_k (L)$.
(Actually, these results are equivalent to each other by using the result of Rubinstein-Tian-Zhang \cite[Theorem 2.8]{RTZ}.)
By taking the limit (for more details, see \cite{BJ,FO}), we obtain $\delta (L) = \lim_k \delta_k (L) \geq \gamma (L)$, so K.Zhang's result \cite[Theorem 2.4]{Zha} implies that the corollary above holds.
Note that our approach in this paper doesn't use the delta-invariant $\delta(L)$.
\end{rem}

Secondly, we consider the singular case.
We refer the works by K.Zheng \cite{Zhe2,Zhe} to the readers for the details of {\kah} metrics with cone singularities.
Assume that $D$ is a smooth hypersurface in $X$ and $s_D$ is a defining section of $D$.
We take $b \in (0,1)$.
Fix a smooth Hermitian metric $| \cdot |$ on the line bundle $\mathscr{O} (D)$ associated with $D$.
We set $\eta_b := (1-b) \sqrt{-1} \Theta (\mathscr{O} (D), | \cdot |)$.
By scaling, we can define the following  singular probability measure by
$$
dV_{f_b} = f_b dV := | s_D |^{-2(1-b)} dV.
$$
Note that there exists $p>1$ such that $f_b \in L^p (dV)$.
In this case, note that we can also write $\gamma_{k, f_b}(L)$ as the log canonical threshold for the pair $\left(X^N , (1-b) \sum_{i=1}^{N} \pi_i^* D \right)$, i.e.,
$$
\gamma_{k, f_b} (L) = {\rm lct}_{\Delta} \left( X^N , (1-b) \sum_{i=1}^{N} \pi_i^* D \, ; \frac{1}{k} ( \det S^{(k)} =0)\right),
$$
where $\pi_i : X^N \to X$ is the $i$-th projection (see also Gibbs stability for pairs in \cite[Section 5]{Be4}).
In addition, the Mabuchi K-energy $\mathcal{M}_{f_b , \eta_b} $ corresponds to the log Mabuchi K-energy for cone angle $2\pi b$ (see \cite{Zhe}).
As before, we set the following :
$$
\mu_b(L) := \frac{-(K_X + (1-b)D )L^{n-1}}{L^n}, \,\,\, s_b (L) := \sup \left\{ s \in \mathbb{R} \, \middle| \, -(K_X + (1-b)D) - sL >0   \right\}.
$$

By K.Zheng's variational characterization of constant scalar curvature {\kah} cone metrics \cite{Zhe}, we obtain the following corollary.
\begin{cor}
\label{cone cscK}
Assume that $K_X + (1-b) D  +\gamma_{f_b} (L) L$ is ample and $\gamma_{f_b} (L) > n \mu_b (L) - (n-1) s_b (L)$.
Then, there exists a unique constant scalar curvature {\kah} metric with cone singularities along $D$ with angle $2\pi b$.
\end{cor}
As an application of Corollary \ref{cone cscK}, for smooth ample divisors of large degree, we obtain the following result which is an analogue on Theorem 1.10 (or Theorem 5.2) in \cite{AHZ}.
\begin{cor}
\label{large degree}
There exists $m_0 \in \mathbb{Z}_{>0}$ which depends only on $X, L$ and $ b \in (0,1]$ such that for any $m \geq m_0$ and a smooth divisor $D \in |mL|$, there exists a unique constant scalar curvature {\kah} metric with cone singularities along $D$ with angle $2\pi b$.
More precisely, we can find a such integer $m_0$ so that $n \mu (L) - (n-1) s (L) - (1-b) m_0 < 0$ and $K_X + m_0L$ is ample.
\end{cor}
Note that we can find generically a smooth divisor $D \in |mL|$ by taking a sufficiently large $m$ because $L$ is an ample line bundle over $X$.

Finally, we mention the relationship between the invariants $\gamma , \gamma_k , \delta$ and $\delta_k$.
As we said, the inequalities $\delta (L) \geq \gamma(L)$ and $\delta_k (L) \geq \gamma_k(L)$ hold by \cite{Be2,FO}.
However, it is known that the equality $\delta_k (L) = \gamma_k (L)$ doesn't hold in general.
Actually, Fujita \cite[Section 3]{F} and Rubinstein-Tian-Zhang \cite[Corollary 7.2]{RTZ} show that $(X,L) = (\mathbb{P}^1, -K_{\mathbb{P}^1})$ breaks the equality for any {\bf finite} $k$.
However, the following problem which asks the equality in the limit case is still open (even if in the Fano case).
\begin{problem}[\cite{Be1,Be2}]
$$
\delta(L) = \gamma (L) ?
$$
\end{problem}

This paper is organized as follows.
In Section 2, we recall the definitions of energy functionals on the space of {\kah} metrics.
We also recall the variational characterization of constant scalar curvature {\kah} (cone) metrics by the (log) Mabuchi K-energy due to Chen-Cheng \cite{CC2}, Berman-Darvas-Lu \cite{BDL2} and K.Zheng \cite{Zhe}.
In Section 3, we prove Theorem \ref{main theorem}.
The outline of the proof is same as the proof in \cite{Be2}, but we need some modifications in order to deal with the singular case.
We also give the proof of Corollary \ref{large degree}.

\begin{ack}
The author is grateful to the organizers of the Study Sessions for Experts ``Canonical {\kah} metrics on {\kah} manifolds and related topics" for giving an opportunity to study recent developments in {\kah} geometry.
He also would like to thank Yuji Odaka for helpful comments and Yoshinori Hashimoto for suggesting that Corollary \ref{cone cscK} can be applied to Corollary \ref{large degree}.
This work was supported by JSPS KAKENHI (the Grant-in-Aid for Research Activity Start-up) Grant Number JP23K19020.
\end{ack}


\section{Energy functionals and coercivity}

In this section, we recall the definitions of energy functionals on the space of {\kah} metrics.
We also recall the results by Chen-Cheng \cite{CC2}, Berman-Darvas-Lu \cite{BDL2} and K.Zheng \cite{Zhe} which characterize the existence of constant scalar curvature {\kah} (cone) metrics by some variational properties of the (log) Mabuchi K-energy.

Let $X$ be a projective manifold and $L$ be an ample line bundle over $X$.
We fix a Hermitian metric $h$ on $L$ whose curvature form defines a {\kah} metric $\omega$, i.e., $\omega = \sqrt{-1} \Theta (L,h)$.
The symbol $\mathcal{H}(L)$ is the space of smooth {\kah} metrics in $c_1 (L)$, i.e., $\mathcal{H}(L) = \{ \phi \in C^\infty (X, \mathbb{R}) \, | \, \omega_\phi = \omega + \dol \phi > 0\}$.
We write the volume of $(X,\omega)$ as $V$, i.e., $V = \int_X \omega^n$.
We fix a smooth probability measure $dV$ on $X$. 
For $0 \leq f \in L^p (dV)$, we set 
$$
dV_f := fdV .
$$
By scaling, we may assume that $dV_f$ is a probability measure, i.e., $\int_X f dV = 1$.

\begin{definition}

We take a smooth $d$-closed (1,1)-form $\chi$ on $X$ and write $\underline{\chi} = \frac{n \int_X \chi \wedge \omega^{n-1}}{\int_X \omega^{n}}$.
We define the following energy functionals on $\mathcal{H} (L)$.
\begin{itemize}

\item[]  $\displaystyle  E (\phi) := \frac{1}{V (n+1)} \sum_{j=0}^{n} \int_X \phi \omega^j \wedge \omega_{\phi}^{n-j} $ \,\, (Aubin-Mabuchi energy)

\vspace{-5pt}
\item[]  $\displaystyle \mathcal{J}_\chi (\phi) := \frac{1}{V} \sum_{j=1}^{n} \int_X \phi \chi \wedge\omega^{j-1} \wedge \omega_{\phi}^{n-j} - \frac{\underline{\chi}}{V (n+1)} \sum_{j=0}^{n} \int_X \phi \omega^j \wedge \omega_{\phi}^{n-j}$ \,\, ($\mathcal{J}_\chi$-functional)

\vspace{-5pt}
\item[]  $\displaystyle J (\phi) := \frac{1}{V} \int_X \phi \omega^n - E (\phi)  $ \,\, ($J$-functional)

\end{itemize}

\end{definition}
We can easily check that $\mathcal{J}_\chi$ is linear with respect to $\chi$, i.e., for any $a, b \in \mathbb{R}$ and $d$-closed smooth (1,1)-forms $\chi_1, \chi_2$, we have $\mathcal{J}_{a\chi_1 + b \chi_B} = a\mathcal{J}_{\chi_1} + b\mathcal{J}_{\chi_2}$.
(For more details of the $\mathcal{J}_\chi$-functional, we refer \cite{Gao,CS,We} to the readers.)
We need the following elementary lemma.

\begin{lemma}
\label{J1}
$$
J_{\omega} (\phi) =  E (\phi) -  \frac{1}{V} \int_X \phi \omega_\phi^n.
$$
\end{lemma}
\begin{proof}
We set $\chi = \omega$.
Since, $\underline{\chi } = n$, we can compute as follows :
\begin{eqnarray*}
\mathcal{J}_{\omega} (\phi) &=& \frac{1}{V} \sum_{j=1}^{n} \int_X \phi \omega \wedge\omega^{j-1} \wedge \omega_{\phi}^{n-j} - \frac{n}{V (n+1)} \sum_{j=0}^{n} \int_X \phi \omega^j \wedge \omega_{\phi}^{n-j}\\
&=& \frac{n+1}{V(n+1)} \sum_{j=1}^{n} \int_X \phi \omega^j \wedge \omega_{\phi}^{n-j} - \frac{n}{V (n+1)} \sum_{j=0}^{n} \int_X \phi \omega^j \wedge \omega_{\phi}^{n-j}\\
&=& \frac{1}{V (n+1)} \sum_{j=0}^{n} \int_X \phi \omega^j \wedge \omega_{\phi}^{n-j} - \frac{1}{V} \int_X \phi \omega_\phi^n \\
&=& E (\phi) -  \frac{1}{V} \int_X \phi \omega_\phi^n.
\end{eqnarray*}
\end{proof}
In addition, the growth of the $\mathcal{J}_\omega$-functional is given by the following lemma.
\begin{lemma}{\rm (\cite[Proposition 22]{CS})}
\label{Jcoer}
There exist constants $A, B >0$ such that
$$
\mathcal{J}_{\omega} (\phi) \geq A J (\phi) - B, \,\,\,\, \forall \phi \in \mathcal{H}(L).
$$
\end{lemma}

In other words, the functional $\mathcal{J}_\omega$ is automatically coercive (see Definition \ref{coer} below).
In the proof of Theorem \ref{main theorem}, we will use the following inequality.
\begin{lemma}{\rm (\cite[Lemma 2.1]{Be2})}
\label{jen}
There exists a constant $C>0$ such that
$$
-E (\phi) + \sup_X \phi \leq J(\phi) + C, \,\,\,\, \forall \phi \in \mathcal{H}(L).
$$
\end{lemma}

The following functionals play a fundamental role in this paper.

\begin{definition}[\cite{Ma,Zhe}]

Fix $\gamma \in \mathbb{R}$.
We define the following functionals on $\mathcal{H}(L)$.
\begin{itemize}

\item[]  $\displaystyle   {\rm Ent}_{dV_f}(V^{-1}\omega_\phi^n) := \int_X \log \left( \frac{V^{-1}\omega_\phi^n}{dV_f} \right)  V^{-1}\omega_\phi^n$ \,\, (relative entropy)

\vspace{-5pt}
\item[]  $\displaystyle \mathcal{M}_{f, \eta} (\phi) := {\rm Ent}_{dV_f}(V^{-1}\omega_\phi^n) + \mathcal{J}_{-{\rm Ric}dV + \eta} (\phi) $ \,\, (Mabuchi K-energy for $f$ and $\eta$)

\vspace{-5pt}
\item[]  $\displaystyle \mathcal{D}_{-\gamma, f} (\phi) := - E (\phi) - \frac{1}{\gamma} \log \int_X e^{-\gamma \phi} dV_f $ \,\, (twisted Ding functional)

\end{itemize}
\end{definition}

\begin{rem}
The case when $f=1$ and $\eta=0$ corresponds to the smooth case since the Mabuchi K-energy for $1$ and $0$ is exactly equal to the Chen-Tian formula of the Mabuchi functional $\mathcal{M}$ (see \cite{Chen1,Ma,Ti2}) :
\begin{equation*}
\mathcal{M} (\phi) = {\rm Ent}_{dV} (V^{-1} \omega_\phi) + J_{-{\rm Ric}\, dV} (\phi).
\end{equation*}
In this case, a critical point of $\mathcal{M}$ is a constant scalar curvature {\kah} metric.

In addition, as we said in Introduction, the case when $f_b = |s_D|^{-2(1-b)}$ and $\eta_b = (1-b) \sqrt{-1} \Theta (\mathscr{O} (D), | \cdot |)$ corresponds to the conic case (see the log Mabuchi K-energy in \cite{Zhe}).
\end{rem}

It is well-known that the relative entropy satisfies the following
\begin{equation}
\label{ent legendre}
{\rm Ent}_\mu (\nu) = \sup_{a \in C^0 (X)} \left( \int_X a \nu - \log \int_X e^a \mu \right)
\end{equation}
for any two probability measures $\mu, \nu$ (see \cite[Proposition 2.10]{BBEGZ}).
In other words, the relative entropy is the Legendre transformation of the functional $\phi \to \log \int_X e^\phi$.

In order to compare the Mabuchi K-energy and the twisted Ding functional, we show the following lemma which is a generalization of the inequality in \cite[Theorem 3.4]{Be1}.
\begin{prop}
\label{Mabuchi and Ding}
For $\gamma \in \mathbb{R}$, we have
$$
\mathcal{M}_{f,\eta}(\phi) \geq \gamma \mathcal{D}_{-\gamma,f} (\phi) + J_{-{\rm Ric} \, dV + \eta + \gamma \omega } (\phi).
$$
\end{prop}

\begin{proof}

We can compute as follows by choosing $a = - \gamma \phi $ in the equation (\ref{ent legendre}).
\begin{eqnarray*}
\mathcal{M}_{f , \eta}(\phi)
&\geq& - \log \int_X e^{-\gamma \phi} dV_f - \gamma \int_X \phi V^{-1} \omega_\phi^n + J_{-{\rm Ric}\,dV + \eta } (\phi) \\
&=& \gamma \mathcal{D}_{-\gamma, f} (\phi) + \gamma \left( E(\phi) - \frac{1}{V} \int_X \phi \omega_\phi^n \right) + J_{-{\rm Ric}\,dV + \eta} (\phi)\\
&=& \gamma \mathcal{D}_{-\gamma, f} (\phi) + \gamma J_{\omega} (\phi) + J_{-{\rm Ric} \, dV + \eta } (\phi)\\
&=& \gamma \mathcal{D}_{-\gamma, f} (\phi) + J_{-{\rm Ric} \, dV + \eta + \gamma \omega } (\phi).
\end{eqnarray*}
Here we have used Lemma \ref{J1}.
\end{proof}

We recall the following asymptotic property of functionals, which plays a very important role on the study of canonical {\kah} metrics.
\begin{definition}
\label{coer}
A functional $F : \mathcal{H} (L) \to \mathbb{R}$ is said to be coercive if there exist constants $C_1, C_2 > 0$ such that
$$
F (\phi) \geq C_1 J (\phi) -C_2 , \,\,\,\, \forall \phi \in \mathcal{H} (X ,\omega).
$$
\end{definition}

We recall the following result which shows that the coercivity of the Mabuchi K-energy characterize the existence of constant scalar curvature {\kah} metrics.
\begin{thm}[\cite{BDL2,CC2}]
There exists a unique constant scalar curvature {\kah} metric in $c_1(L)$ if and only if the Mabuchi K-energy is coercive.
\end{thm}
The conic version of the result above is also given.
\begin{thm}[\cite{Zhe}]
Fix $b \in (0,1]$ and a smooth divisor $D$ on $X$.
There exists a unique constant scalar curvature {\kah} cone metric in with cone angle $2 \pi b$ if and only if the log Mabuchi K-energy for cone angle $2\pi b$ is coercive.
\end{thm}

\begin{rem}
In \cite{BDL2,CC2,Zhe}, they consider the metric completion of $\mathcal{H} (L)$ which is called the finite energy space $\mathcal{E}^1 (X,\omega)$ (\cite{GZ}, see also \cite[Section 3]{Da}) on the study of (log) Mabuchi K-energy.
(It is known that the space of {\kah} cone metrics is included in $\mathcal{E}^1 (X,\omega)$.)
The corresponding metric on $\mathcal{H} (L)$ is $d_1$-metric whose growth is compatible with $J$-functional (see \cite{Da}).
Note that the coercivity on $\mathcal{E}^1 (X,\omega)$ is equivalent to the coercivity on $\mathcal{H}(L)$ by applying the regularization of the relative entropy \cite[Lemma 3.1]{BDL1} and the fact that the functionals $\mathcal{J}_\chi$ and $ J$ are continuous with respect to $d_1$-metric \cite[Lemma 5.23]{DR}.
\end{rem}


\section{Proof}

In this section, we prove Theorem \ref{main theorem}.
The proof in this paper is similar to the proof in \cite{Be2}, but we need some technical modifications in order to deal with singular density.

For $k \in \mathbb{Z}_{>0}$, we fix a basis $s_{i}^{(k)}, (i =1,2,...,N)$ of $H^0 (X , L^k)$, where $N = N_k := \dim H^0 (X , L^k) $.
For $\phi \in \mathcal{H}(L)$ and a probability measure $\mu$, we define the Hermitian matrix $H^{(k)} (\phi, \mu)$ by
$$
H^{(k)} (\phi, \mu) := \left( \int_{X} \left( s_{i}^{(k)}, s_{j}^{(k)}\right)_{(he^{-\phi})^k} \mu \right)_{ij}.
$$
We set
$$
E_{k } (\phi) := - \frac{1}{kN} \log \det H^{(k)} (\phi, dV).
$$
\begin{definition}
The $k$-th approximate Ding functional is defined by
$$
\mathcal{D}_{k, -\gamma, f} (\phi) := - E_{k} (\phi) - \frac{1}{\gamma} \log \int_X e^{-\gamma \phi} dV_f.
$$
\end{definition}

\begin{rem}
Note that the functional $E_k$ doesn't depend on the parameter $\gamma$ since we consider the fixed smooth probability measure $dV$.
So, the $k$-th approximate Ding functional $\mathcal{D}_{k, -\gamma, f} $ above is different from the approximate Ding functional in \cite{Be2}. 
\end{rem}

Firstly, we show the following proposition which is a modification of Proposition 2.3 in \cite{Be2}.
\begin{prop}
\label{approx Ding}
Fix $\tau > 0$ and $\gamma > 0$.
We take a positive integer $k$ such that $f \in L^{k\tau / (k\tau-\gamma(1+\tau))} (dV)$ and $k\tau/\gamma (1+\tau)  > 1$ .
Then, for any $\phi \in \mathcal{H}(L)$, we have
\begin{eqnarray*}
&& \hspace{-40pt} -\frac{1}{\gamma(1+\tau) N} \log \mathcal{Z}_{N, f} (-\gamma(1+\tau) )\\
&& \leq \mathcal{D}_{k,-\gamma,f} (\phi ) + \frac{1}{k} \log N + \frac{k\tau-\gamma(1+\tau)}{\gamma k (1+\tau)} \log \int_{X} f^{k\tau / (k\tau-\gamma(1+\tau))}  dV. 
\end{eqnarray*}
\end{prop}

\begin{proof}
Fix $\phi \in \mathcal{H} (L)$.
Then, we can compute as follows :
\begin{eqnarray*}
\mathcal{Z}_{N, f} (-\gamma)
&=& \int_{X^N} \left| \det S^{(k)} \right|^{-2\gamma/k}_{h^k} dV_f^{\otimes N}\\
&=& \int_{X^N} \left| \det S^{(k)} \right|^{-2\gamma/k}_{(h e^{-  \phi})^k } (e^{-\gamma \phi} dV_f)^{\otimes N}.
\end{eqnarray*}

Firstly, we assume that $e^{-\gamma \phi} dV_f$ is a probability measure.
Then, the {\hol} inequality implies that
\begin{eqnarray*}
1&=& \int_{X^N}  (e^{-\gamma \phi} dV_f)^{\otimes N} \\
&=& \int_{X^N}  \left( \left| \det S^{(k)} \right|^{-2\gamma/k}_{(h e^{-  \phi})^k } \prod^N e^{-\tau\gamma \phi/(1 + \tau)}  \right) \left( \left|  \det S^{(k)} \right|^{2\gamma/k}_{(h e^{-  \phi})^k }  \prod^N e^{\tau\gamma \phi/(1 + \tau)} \right) (e^{-\gamma \phi} dV_f)^{\otimes N}  \\
&\leq& \left( \int_{X^N} \left| \det S^{(k)} \right|^{-2\gamma (1+\tau)/k}_{(h e^{-  \phi})^k }  \prod^N e^{-\tau\gamma \phi}  (e^{-\gamma \phi} dV_f)^{\otimes N} \right)^{1 /(1+\tau)} \\
&& \hspace{50pt}\times \left( \int_{X^N} \left| \det S^{(k)} \right|^{2\gamma (1 + \tau)/k\tau}_{(h e^{-  \phi})^k } \prod^N e^{\gamma \phi}  (e^{-\gamma \phi} dV_f)^{\otimes N} \right)^{\tau /(1 + \tau)}\\
&\leq& \left( \int_{X^N} \left| \det S^{(k)} \right|^{-2\gamma (1+\tau)/k}_{(h e^{-  \phi})^k }  (e^{-(1+\tau)\gamma \phi} dV_f)^{\otimes N} \right)^{1 /(1+\tau)} \hspace{-17pt} \times \left( \int_{X^N} \left| \det S^{(k)} \right|^{2\gamma (1 + \tau)/k\tau}_{(h e^{-  \phi})^k } dV_f^{\otimes N} \right)^{\tau /(1 + \tau)}.
\end{eqnarray*}
Note that the first term in the last inequality is $\mathcal{Z}_{N,f} (-\gamma(1+\tau))^{1 /(1+\tau)}$.
For simplicity, we write the product of the density functions as
$$
F(x_1 ,..., x_N) := \prod_i f(x_i) \in L^p (dV^{\otimes N}).
$$

Recall that $dV$ is a (fixed) smooth probability measure.
By using the {\hol} inequality again and the assumption that $k\tau/\gamma (1+\tau)  > 1$, we have
\begin{eqnarray*}
&&\hspace{-10pt} \int_{X^N} \left| \det S^{(k)} \right|^{2\gamma (1 + \tau)/k\tau}_{(h e^{-  \phi})^k } dV_f^{\otimes N}\\
&&= \int_{X^N} \left| \det S^{(k)} \right|^{2\gamma (1 + \tau)/k\tau}_{(h e^{-  \phi})^k } F dV^{\otimes N}\\
&&\leq \left( \int_{X^N} \left| \det S^{(k)} \right|^{2}_{(h e^{-  \phi})^k } dV^{\otimes N} \right)^{\gamma (1 + \tau) /k\tau} \left( \int_{X^N} F^{k\tau / (k\tau-\gamma(1+\tau))} dV^{\otimes N} \right)^{(k\tau-\gamma(1+\tau))/k\tau} \\
&&= \left\Vert \det S^{(k)} \right\Vert^{\gamma (1 + \tau) /2k\tau} _{L^2(h e^{-  \phi} ,dV^{\otimes N})} \left( \int_{X^N} F^{k\tau / (k\tau-\gamma(1+\tau))} dV^{\otimes N} \right)^{(k\tau-\gamma(1+\tau))/k\tau} .
\end{eqnarray*}

Thus, we obtain the following inequality.
\begin{eqnarray*}
&&\hspace{-20pt} -\frac{1}{\gamma(1+\tau) } \log \mathcal{Z}_{N, f}(-\gamma ( 1+\tau) ) \\
&&\leq \frac{1}{k} \log \left\Vert \det S^{(k)}  \right\Vert^{2}_{L^2 ( h e^{-  \phi},  dV^{\otimes N})} + \frac{k\tau-\gamma(1+\tau)}{\gamma k (1+\tau)} \log \int_{X^N} F^{k\tau / (k\tau-\gamma(1+\tau))}  dV^{\otimes N}.
\end{eqnarray*}

Secondly, we consider the general case.
If the measure $e^{-\gamma \phi} dV_b$ is not a probability measure, by considering the normalization
$$
\phi \mapsto \phi + \frac{1}{\gamma}  \log \int_X e^{-\gamma \phi} dV_f,
$$
we have
\begin{eqnarray*}
&&\hspace{-20pt}-\frac{1}{\gamma(1+\tau)} \log \mathcal{Z}_{N,f} (-\gamma(1+\tau))\\
&&\hspace{-10pt}\leq \frac{1}{k} \log \left\Vert \det S^{(k)}  \right\Vert^{2}_{L^2 (h e^{-  \phi},  dV^{\otimes N})} - \frac{N}{ \gamma } \log \int_X e^{-\gamma \phi} dV_f + \frac{k\tau-\gamma(1+\tau)}{\gamma k (1+\tau)} \log \int_{X^N} F^{k\tau / (k\tau-\gamma(1+\tau))}  dV^{\otimes N}. \\
\end{eqnarray*}
By the formula in Berman-Boucksom \cite[Lemma 4.3]{BB} :
$$
\left\Vert \det S^{(k)}  \right\Vert^{2}_{L^2 ( h e^{-  \phi},  dV^{\otimes N})} = N ! \det H^{(k)} (\phi,  dV),
$$
we have
\begin{eqnarray*}
&& \hspace{-30pt} -\frac{1}{\gamma(1+\tau) N } \log \mathcal{Z}_{N, f} (-\gamma(1+\tau)) \\ 
&&\leq \frac{1}{N k} \log \det H^{(k)} (\phi, dV)  - \frac{1}{ \gamma} \log \int_X e^{-\gamma \phi} dV_f + \frac{1}{Nk} \log N! \\
&&\hspace{40pt}+ \frac{k\tau-\gamma(1+\tau)}{\gamma kN (1+\tau)} \log \int_{X^N} F^{k\tau / (k\tau-\gamma(1+\tau))}  dV^{\otimes N} \\
&&\leq -E_{k} (\phi)  - \frac{1}{ \gamma} \log \int_X e^{-\gamma \phi} dV_f + \frac{1}{k} \log N + \frac{k\tau-\gamma(1+\tau)}{\gamma k (1+\tau)} \log \int_{X} f^{k\tau / (k\tau-\gamma(1+\tau))}  dV \\
&&= \mathcal{D}_{k,-\gamma, f} (\phi ) + \frac{1}{k} \log N + \frac{k\tau-\gamma(1+\tau)}{\gamma k (1+\tau)} \log \int_{X} f^{k\tau / (k\tau-\gamma(1+\tau))}  dV .  
\end{eqnarray*}
Here, we have used the Fubini theorem for $F (x) = \prod_{i=1}^{N} f (x_i) $ and the Stirling formula $\log N! \approx N \log N - N$.
\end{proof}

In order to compare the twisted Ding functional $\mathcal{D}_{-\gamma, f}$ and the $k$-th approximate Ding functional $\mathcal{D}_{k, -\gamma,f}$, we need the lemma below.
We fix $k \in \mathbb{Z}_{>0}$ and $c \in \mathbb{R}_{>0}$ such that $c \omega + {\rm Ric} \, dV \geq 0$.
We set $\epsilon := c k^{-1}$ and $\phi^{(\epsilon)} := (1-\epsilon) \phi$.

\begin{lemma}
\label{energy}
There exists $C_0$ depends only on $\omega$ such that
$$
-\frac{1}{1-\epsilon}  E_{k} \left( \phi^{(\epsilon)} \right) \leq - E (\phi) + C_0 k^{-1} \left( - E (\phi)  + \sup_X \phi \right).
$$

\end{lemma}

\begin{rem}
The proof of this lemma is essentially same as the proof in \cite[Lemma 2.4]{Be2}.
In our case, we consider the energy functional $E_k $ defined by the {\bf fixed} smooth probability measure $dV$, i.e., independent of $\gamma$ (compere the functional $\mathcal{E}_k$ in \cite{Be2}).
In this paper, we only check the positivity of some curvature form along a weak geodesic in order to apply the result of positivity of direct image sheaves \cite{Bernd}.
\end{rem} 

\begin{proof}
We check the positivity of the curvature form of some Hermitian metric on $kL - K_X$.
Let $\psi_t$ be a weak geodesic connecting $0$ and $\phi$ (see \cite{Chen} and \cite[Section 3]{Da}), so $\psi_t$ is a $p^* \omega$-plurisubharmonic function on $X \times [0,1]\times i \mathbb{R}$ which is independent of the imaginary part.
Here, $p : X \times [0,1]\times i \mathbb{R} \to X$ is the projection.
Since $\det H \left( \psi_t^{(\epsilon)}, dV \right)$ can be identified with an $L^2$-metric on $\det H^0 (X,kL  )$, if the metric $( h e^{- \psi_t^{(\epsilon)} } )^k  dV$ has a positive curvature, then the function $t \mapsto E_k ( \psi_t^{(\epsilon)})$ is convex by \cite{Bernd}.
From the assumption that $\epsilon = c k^{-1}$, we can compute as follows :
\begin{eqnarray*}
&& \Theta \left( p^*(kL - K_X), ( h e^{- \psi_t^{(\epsilon)} } )^k  dV \right) \\
&=& k (p^*\omega + \dol \psi_t^{(\epsilon)}) + p^*{\rm Ric}\, dV \\
&=& k(1-\epsilon) ( p^*\omega  + \dol \psi_t)+ p^* (c \omega + {\rm Ric}\, dV) \geq 0.
\end{eqnarray*}
Thus, the function
$$-\frac{1}{1-\epsilon}  E_{k} \left( \psi_{t}^{(\epsilon)} \right) + E (\psi_t)
$$
is concave in the variable $t$ because the Aubin-Mabuchi energy $E$ is affine along weak geodesics.
The rest of the proof is completely same as the proof of \cite[Lemma 2.4]{Be2}.
\end{proof}

We give the following inequality between the $k$-th approximate Ding functional and the twisted Ding functional.

\begin{prop}
\label{twist and approx}
Set $\gamma^{(\epsilon)} := (1-\epsilon)\gamma$.
There exists a positive constant $C_1$ such that
$$
\gamma  \mathcal{D}_{k,-\gamma, f} (\phi ) \leq  \gamma^{(\epsilon)} \left( \mathcal{D}_{-\gamma^{(\epsilon)}, f} (\phi) + C_1 k^{-1} \mathcal{J}_{\omega} (\phi)  + C_1 \right)
$$
for any $\phi \in \mathcal{H}(L)$.
\end{prop}

\begin{proof}

By applying the inequality in Proposition \ref{approx Ding} for $\phi^{(\epsilon)} = (1 - \epsilon) \phi$ and Lemma \ref{energy}, we can compute as follows :
\begin{eqnarray*}
&&\hspace{-50pt}\gamma \mathcal{D}_{k,-\gamma,f} (\phi^{(\epsilon)} ) \\
&=& \gamma  \left( -E_{k} (\phi^{(\epsilon)})  - \frac{1}{ \gamma} \log \int_X e^{-\gamma \phi^{(\epsilon)}} dV_f \right) \\
&=&  \gamma^{(\epsilon)} \left( - \frac{1}{1 - \epsilon}E_{k} (\phi^{(\epsilon)})  - \frac{1}{ \gamma^{(\epsilon)}} \log \int_X e^{-\gamma^{(\epsilon)} \phi} dV_f \right) \\
&\leq& \gamma^{(\epsilon)} \left(  - E (\phi)   - \frac{1}{ \gamma^{(\epsilon)}} \log \int_X e^{-\gamma^{(\epsilon)} \phi} dV_f +C_0 k^{-1} \left( - E (\phi) + \sup_X \phi \right) \right)\\
&\leq& \gamma^{(\epsilon)} \left( \mathcal{D}_{-\gamma^{(\epsilon)}, f} (\phi) + C_1 k^{-1} \mathcal{J}_{\omega} (\phi)  + C_1 k^{-1} \right).
\end{eqnarray*}
Here, we have used the combination of Lemma \ref{Jcoer} and Lemma \ref{jen} in the last inequality.
\end{proof}

Finally, we give the proof of the main result in this paper and Corollary \ref{large degree}.

\medskip

\hspace{-21pt} {\it Proof of Theorem \ref{main theorem}.}

\medskip

Recall that $\gamma^\prime = \gamma (1-ck^{-1})(1-C_1 k^{-1})= \gamma^{(\epsilon)} (1 - C_1 k^{-1})$.
Proposition \ref{approx Ding} and Proposition \ref{twist and approx} imply that
\begin{eqnarray*}
&&\hspace{-25pt}-\frac{1}{N(1+\tau)} \log \mathcal{Z}_{N, f} (-\gamma(1+\tau)) +  \mathcal{J}_{-{\rm Ric} \, dV + \eta + \gamma^\prime \omega } (\phi)\\
&&\leq  \gamma^{(\epsilon)} \mathcal{D}_{- \gamma^{(\epsilon)}, f} (\phi) +  \gamma^{(\epsilon)} C_1 k^{-1} \mathcal{J}_{\omega} (\phi)  + \gamma^{(\epsilon)}C_1 + \mathcal{J}_{-{\rm Ric} \, dV + \eta + \gamma^\prime \omega } (\phi) \\
&& \hspace{20pt}+ \frac{k\tau-\gamma(1+\tau)}{ k(1+\tau)} \log \int_{X} f^{k\tau / (k\tau-\gamma(1+\tau))}  dV+ \frac{\gamma}{k} \log N\\
&&= \gamma^{(\epsilon)} \mathcal{D}_{- \gamma^{(\epsilon)}, f} (\phi) + \mathcal{J}_{-{\rm Ric} \, dV + \eta + \gamma^{(\epsilon)} \omega } (\phi)  + \gamma^{(\epsilon)}C_1  \\
&&\hspace{20pt} + \frac{k\tau-\gamma(1+\tau)}{ k(1+\tau)} \log \int_{X} f^{k\tau / (k\tau-\gamma(1+\tau))}  dV+ \frac{\gamma}{k} \log N\\
&&\leq  \mathcal{M}_{f,\eta}( \phi) + \gamma^{(\epsilon)}C_1 + \frac{k\tau-\gamma(1+\tau)}{ k(1+\tau)} \log \int_{X} f^{k\tau / (k\tau-\gamma(1+\tau))}  dV+ \frac{\gamma}{k} \log N. 
\end{eqnarray*}
Here, we use the Proposition \ref{Mabuchi and Ding} in the last inequality
Thus, we have finished the proof of Theorem \ref{main theorem}.\sq

\medskip

\hspace{-21pt} {\it Proof of Corollary \ref{large degree}.}

\medskip

For simplicity, we write $s_{F} (L) := \{ s \in \mathbb{R} \, |  \, F -sL >0  \} $ for a divisor $F$.
We can easily check that $s_{F_1} (L) + s_{F_2} (L) \leq s_{F_1 + F_2} (L)$ for every $F_1, F_2$.
So, we have
\begin{eqnarray*}
s_b (L) = s_{-(K_X + (1-b)D)}(L) \geq s_{-K_X} (L) + s_{- (1-b) D} (L).
\end{eqnarray*}
Since $D \in |mL|$ and the fact that $s_{aL} (L) = a$ for any $a\in \mathbb{R}$, we have
\begin{eqnarray*}
s_b (L) \geq  s_{-K_X} (L) + s_{- (1-b) mL} (L) =  s_{-K_X} (L) - (1-b) m.
\end{eqnarray*}
By the equality
$$
\mu_b (L) = \frac{-(K_X + (1-b)D)L^{n-1}}{L^n}=\frac{-(K_X + (1-b)mL)L^{n-1}}{L^n} = \frac{-K_X L^{n-1}}{L^n} - m (1-b),
$$
we obtain the following inequality
\begin{eqnarray*}
n \mu_b (L) - (n-1) s_b (L)
&\leq&n \frac{-K_X L^{n-1}}{L^n}  -  (n-1) s_{-K_X} (L) - (1-b) m \\
&=&n \mu (L) - (n-1) s (L) - (1-b) m.
\end{eqnarray*}
So, we can find a sufficiently large integer $m_0$ so that $n \mu (L) - (n-1) s (L) - (1-b) m_0 < 0$ and $K_X + m_0L$ is ample.
Since $\gamma_b (L) > 0$ (in fact, $\gamma_{f_b}(L)$ is bounded below by the $\alpha$-invariant (see \cite{Be2})), we have finished the proof of Corollary \ref{large degree} by using Corollary \ref{cone cscK}.
(Note that $\gamma_{f_b}(L)$ depends on the choice of $m$ and $D \in |mL|$.)
\sq


\bigskip
\address{
National Institute of Technology (KOSEN),\\
Wakayama College,\\
77, Nojima, Nada-chou, Gobo-shi\\
Wakayama, 644-0023\\
Japan
}
{aoi@wakayama-nct.ac.jp \,\, {\it or} \,\, takahiro.aoi.math@gmail.com }

\end{document}